\documentclass[12pt]{amsart}
\usepackage{amssymb}
\usepackage{bbm}
\usepackage{mathtools}
\usepackage{enumitem}
\usepackage{tikz-cd}
\usepackage{calc}

\newcommand{\textover}[3][l]{\makebox[\widthof{#3}][#1]{#2}}

\newcommand{\matdot}{\mathrm{.}}
\newcommand{\matcom}{\mathrm{,}}

\newcommand{\id}{\mathrm{id}}

\newcommand{\funCD}{\overline{\textover[c]{$\cdot$}{F}}}

\newcommand{\cleft}{\mathrm{C}}
\newcommand{\cleftcl}{\mathrm{Cleft}}

\newcommand{\scr}[1]{\mathcal{#1}}

\newcommand{\ndN}{\mathbb{N}}

\newcommand{\fK}{\mathbbm{k}}

\newcommand{\ot}{\otimes}
\newcommand{\ol}{\overline}

\newcommand{\catunit}{\mathbbm{1}}
\newcommand{\dcatunit}{\ol{\catunit}}

\newcommand{\conv}{*}
\newcommand{\coninv}[1]{#1^{-}}
\newcommand{\coinv}[1]{ {{#1}}^{\mathrm{co} \ol{H}}}
\newcommand{\bimu}{\mu}
\newcommand{\bicomu}{\Delta}
\newcommand{\biunit}{\eta}
\newcommand{\bicounit}{\varepsilon}
\newcommand{\antip}{S}
\newcommand{\dantip}{\ol{\antip}}
\newcommand{\mact}{\nu}
\newcommand{\coact}{\rho}
\newcommand{\brd}{c}
\newcommand{\sect}{\gamma}
\newcommand{\dbimu}{\ol{\mu}}
\newcommand{\dbicomu}{\ol{\bicomu}}
\newcommand{\dbiunit}{\ol{\biunit}}
\newcommand{\dbicounit}{\ol{\bicounit}}
\newcommand{\dbrd}{\brd_{H,\ol{H}}}
\newcommand{\sigmul}{\bimu_\sigma}
\newcommand{\sigsmashp}[2]{ {#1} \#_{{#2}} \ol{H} }
\newcommand{\sigsmash}[1]{ \sigsmashp{ {#1} }{\sigma} }

\newcommand{\cocyclesetR}{Z(R)}
\newcommand{\cocyclesetH}{Z(\scr{H})}
\newcommand{\cocyclesetHp}{Z'(\scr{H})}
\newcommand{\cocyclebij}{\Phi}
\newcommand{\cleftbij}{\Psi}

\theoremstyle{plain}
\newtheorem{thm}{Theorem}[section]
\newtheorem*{thm*}{Theorem}
\newtheorem*{classification*}{Classification theorem}
\newtheorem{lem}[thm]{Lemma}
\newtheorem{cor}[thm]{Corollary}

\newtheorem*{notation*}{Notation}
\newtheorem{prop}[thm]{Proposition}

\newtheorem{defi}[thm]{Definition}

\theoremstyle{remark}
\newtheorem{rema}[thm]{Remark}
\newtheorem*{rema*}{Remark}
\newtheorem{exa}[thm]{Example}

\numberwithin{equation}{section}

\newtagform{roman}[]()

\begin{document}

\title[Two-cocycles and cleft extensions]{Two-cocycles and cleft extensions in left braided categories}
\author{Istv\'{a}n Heckenberger, Kevin Wolf}

\begin{abstract}
	We define two-cocycles and cleft extensions in categories that are not necessarily braided, but where specific objects braid from one direction, like for a Hopf algebra $H$ a Yetter-Drinfeld module braids from the left with $H$-modules.
	We will generalize classical results to this context and give some application for the categories of Yetter-Drinfeld modules and $H$-modules. In particular we will describe liftings of coradically graded Hopf algebras in the category of Yetter-Drinfeld modules with these techniques.
\end{abstract}

\maketitle

\section*{Introduction}
The motivation of this work was to provide an abstract framework and a new description of the set $\cleftcl'(\scr{H})$ appearing in \cite{guidedtour}, section~3. This set arises in the context of finding all liftings of coradically graded Hopf algebras: We have a finite dimensional connected graded Hopf algebra $R \in \prescript{H}{H}{\scr{YD}}$, where $H$ is a finite dimensional cosemisimple Hopf algebra over some field $\fK$, and are interested in Hopf algebras $L$, such that $\mathrm{gr}L \cong R \# H=: \scr{H}$. Therefor it is enough to look at two-cocycles $\sigma : \scr{H} \ot \scr{H} \rightarrow \fK$ and its corresponding cocycle deformation Hopf algebra $\scr{H}_\sigma$, such that $\mathrm{gr}\scr{H}_\sigma \cong \scr{H}$, or equivalently at $\scr{H}$-cleft objects where some corresponding two-cocycle has that property. 

Some results have already been established for describing these cleft objects: 
There exists some $R$-comodule algebra $\scr{E} \in H\mathrm{-mod}$ and a convolution invertible comodule morphism $\gamma: R \rightarrow \scr{E}$ in $H\mathrm{-mod}$, such that $E = \scr{E} \# H$ (\cite{MR3712438}, Lemma 4.1). This means $\scr{E}$ basically has the properties that would describe an $R$-cleft object, except that it is only an object in $H\mathrm{-mod}$ and not in $\prescript{H}{H}{\scr{YD}}$. We could correspond a morphism $\sigma : R \ot R \rightarrow \fK$ in $H\mathrm{-mod}$ to $\gamma$, that then again basically has all the properties of a two-cocycle, except that it is not a morphism in $\prescript{H}{H}{\scr{YD}}$.

This is why we establish a new notion of two-cocycles and cleft extensions in a more general sense, similar to the center construction of a monoidal category (Definitions \ref{def_cocycle} and \ref{def_cleft}). We will see that the known results for two-cocycles and cleft extensions (\cite{MR1765310}, section~1) can be generalized to this context: We can associate a multiplication $\sigmul$ to a two-cocycle $\sigma$ that then is associative (Proposition \ref{prop_sigmul_assoc}) and there is a bijective correspondence between two-cocylces and cleft extensions with fixed section up to isomorphy (Theorem~\ref{thm_cleft_cocycle}).

We will then give an application in the category $\prescript{H}{H}{\scr{YD}}$, where we describe these new cleft objects and two-cocycles by classical cleft objects and two-cocycles (Proposition \ref{prop_cleftcocyclebij}, Corollary \ref{cor_cleftcocyclebij}). From this we can also describe $\cleftcl'(\scr{H})$ (Proposition \ref{cor_cleftcocyclebij2}, Corollary \ref{cor_cleftcocyclebij3}).

\section{Preliminaries}

Let $\scr{C}$ and $\scr{D}$ be strict monoidal categories. We assume that equalizers exist in $\scr{D}$ and denote the unit object of $\scr{C}$ with $\catunit$.

\begin{defi}
	Let $\funCD :\scr{C} \rightarrow \scr{D}$ be a strict monoidal functor. A \textbf{left braiding} $\brd$ for $\funCD$ is a family of isomorphisms $\brd_{X, V} : \ol{X} \ot V \rightarrow V \ot \ol{X}$ in $\scr{D}$, where $X \in \scr{C}$ and $V \in\scr{D}$, such that
	\begin{enumerate}
		\item For objects $X,Y \in \scr{C}$ and $V,W \in \scr{D}$, a morphism $f: X \rightarrow Y$ in $\scr{C}$ and a morphism $g: V \rightarrow W$ in $\scr{D}$ the diagram
		\begin{align*}\begin{tikzcd}[ampersand replacement=\&]
		\ol{X} \ot V \arrow{d}{\brd_{X, V}} \arrow{r}{\ol{f} \ot g} \& \ol{Y} \ot W \arrow{d}{\brd_{Y, W}} \\
		V \ot \ol{X} \arrow{r}{g \ot \ol{f}} \& W \ot \ol{Y}
		\end{tikzcd}\end{align*}
		commutes.
		\item For $X,Y \in \scr{C}$, $V \in \scr{D}$ the diagram
		\begin{align*}\begin{tikzcd}[ampersand replacement=\&]
		\ol{X} \ot \ol{Y} \ot V \arrow[']{dr}{\id \ot \brd_{Y, V} } \arrow{rr}{\brd_{X\ot Y, V}} \& \& V \ot \ol{X} \ot \ol{Y} \\
		\& \ol{X} \ot V \ot \ol{Y} \arrow[']{ur}{\brd_{X, V} \ot \id}
		\end{tikzcd}\end{align*}
		commutes.
		\item For $X \in \scr{C}$, $V,W \in \scr{D}$ the diagram
		\begin{align*}\begin{tikzcd}[ampersand replacement=\&]
		\ol{X} \ot V \ot W \arrow[']{dr}{\brd_{X, V} \ot \id } \arrow{rr}{\brd_{X, V\ot W}} \& \& V \ot W \ot \ol{X} \\
		\& V\ot \ol{X} \ot W \arrow[']{ur}{\id \ot \brd_{X, W}}
		\end{tikzcd}\end{align*}
		commutes.
	\end{enumerate}
\end{defi}
These are the usual axioms that define a braiding in a strict monoidal category. Let $\funCD :\scr{C} \rightarrow \scr{D}$ be a strict monoidal functor and let $\brd$ be a left braiding for $\funCD$.

\begin{rema}
	For objects $X\in \scr{C}$, $V\in \scr{D}$ we have
	\begin{align*}
	\brd_{\catunit, V} = \id_V && \brd_{X, \dcatunit} = \id_{\ol{X}} \matcom
	\end{align*}
	following from axioms (2) and (3) of the definition of a left braiding.
\end{rema}

\begin{exa}
	The center $\scr{C}:= Z(\scr{D})$ of a monoidal category $\scr{D}$ (as described in \cite{MR1321145}, XIII.4) along with the forgetful functor $\scr{C} \rightarrow \scr{D}$ naturally has a left braiding (or right braiding, depending on the definition).
\end{exa}

\begin{exa} \label{exa_yetter_drinfeld}
	Let $\fK$ be a field and $H$ be a Hopf algebra over $\fK$, $\scr{C} =\prescript{H}{H}{\scr{YD}}$ be the category of left-left Yetter-Drinfeld modules over $H$ and $\scr{D}$ be the category of left $H$-modules. The forgetful functor $\scr{C} \rightarrow \scr{D}$, that just forgets the $H$-comodule structure, has a left braiding: For a Yetter-Drinfeld module $X \in \scr{C}$ and an $H$-module $Y \in \scr{D}$ define
	\begin{align*}
	\brd_{X,Y} : X \ot Y \rightarrow Y \ot X, x \ot y \rightarrow x_{(-1)} y \ot x_{(0)} \matdot
	\end{align*}
	If $Y$ is also a Yetter-Drinfeld module, then this braiding is just the usual braiding in the braided category $\prescript{H}{H}{\scr{YD}}$. Showing that $\brd_{X,Y}$ is an isomorphism of $H$-modules and that axioms (1), (2) and (3) hold is completely analogous to showing this in the Yetter-Drinfeld case (for this at no point a comodule structure on $Y$ is needed).
\end{exa}

\begin{notation*}
	Let $\scr{M}$ be a monoidal category. For an object $X \in \scr{M}$ and $n \in \ndN$ we denote $\id^n = \id_X^n := \id_{X^{\ot n}} = \id \ot \cdots \ot \id : X^{\ot n} \rightarrow X^{\ot n}$.
\end{notation*}

\begin{defi}
	We define \textbf{algebras} and \textbf{coalgebras} in a strict mon\-oidal category $\scr{M}$ as usual, as they require no notion of a braiding. If $A \in \scr{M}$ is an algebra, then we denote its multiplication with $\bimu_A$ and its unit with $\biunit_A$.
	If $C \in \scr{M}$ is a coalgebra, then we denote its comultiplication with $\bicomu_C$ and its counit with $\bicounit_C$. Moreover for $n \in \ndN$ we denote \begin{align*}\bimu_A^n := \bimu_A (\bimu_A \ot \id) \cdots (\bimu_A \ot \id^{n-1}) : A^{\ot n+1} \rightarrow A \end{align*} for the $n$-times multiplication. We have $\bimu_A=\bimu_A^1$. Similarly, we denote \begin{align*}\bicomu_C^n := (\bicomu_C \ot \id^{n-1}) \cdots (\bicomu_C \ot \id) \bicomu_C: C \rightarrow  C^{\ot n+1}  \end{align*} for the $n$-times comultiplication, such that $\bicomu_C^1 = \bicomu_C$.
\end{defi}

\begin{rema} \label{rem_tensor_alg_structure}
	Since we are not in a braided category, tensor products of algebras or coalgebras are not naturally algebras or coalgebras again. However combining our two categories $\scr{C}$,$\scr{D}$ and the left braiding $\brd$, some specific tensor products inherit these algebraic structures:
	\begin{itemize}
	\item If $A$ is an algebra in $\scr{D}$ and if $B$ is an algebra in $\scr{C}$, then $A \ot \ol{B}$ is an algebra in $\scr{D}$ where the unit is $\biunit_A \ot \ol{\biunit_B}$ and the multiplication is $(\bimu_A \ot \ol{\bimu_B}) (\id \ot \brd_{B,A} \ot \id)$. 
	\item If $A$ is a coalgebra in $\scr{D}$ and if $B$ is a coalgebra in $\scr{C}$, then $\ol{B}\ot A$ is a coalgebra in $\scr{D}$ with counit $\ol{\bicounit_B} \ot \bicounit_A$ and comultiplication $(\id \ot \brd_{B,A} \ot \id) (\ol{\bicomu_B} \ot \bicomu_A)$.
	\end{itemize}
	In particular if $A,B \in \scr{C}$ are (co)algebras, then $\ol{A}$, $\ol{B}$ and $\ol{A} \ot \ol{B}$ are (co)algebras in $\scr{D}$.
\end{rema}

\begin{defi}
	Let $\scr{M}$ be a monoidal category, $A \in \scr{M}$ be an algebra and $C \in \scr{M}$ a coalgebra. Given two morphisms $f,g : C \rightarrow A$ we define the \textbf{convolution product}
	\begin{align*}
	f \conv g := \bimu_A (f \ot g) \bicomu_C : C \rightarrow A \matdot
	\end{align*}
	We say a morphism $f : C\rightarrow A$ is \textbf{convolution invertible}, if there exists a morphism $g:C \rightarrow A$, such that $f \conv g = \biunit_A \bicounit_C = g \conv f$. In this case we denote $\coninv{f} := g$.
\end{defi}

The convolution product is associative and has unit $\biunit_A \bicounit_C : C \rightarrow A$.

\begin{lem} \label{lem_conv_alg_morphism}
	Let $\scr{M}$ be a monoidal category, $A, A' \in \scr{M}$ be algebras, $C,C' \in \scr{M}$ coalgebras, $f,g : C \rightarrow A$ morphisms in $\scr{M}$, $\psi : A \rightarrow A'$ an algebra morphism and $\phi : C' \rightarrow C$ a coalgebra morphism. Then
	\begin{align*}
		\psi (f \conv g) \phi = (\psi f \phi) \conv (\psi g \phi) \matdot
	\end{align*}
	In particular if $f$ is convolution invertible, hen $ \psi f \phi$ is convolution invertible with inverse $\psi \coninv{f} \phi$.
\end{lem}
\begin{proof}
	The claim directly follows from $\psi \bimu_A = \bimu_{A'} (\psi \ot \psi)$ and $\bicomu_C \phi = (\phi \ot \phi) \bicomu_{C'}$.
\end{proof}

\begin{defi} \label{def_hopf}
	Suppose $H \in \scr{C}$ is an object that is an algebra with multiplication $\bimu$ and unit $\biunit$ and that is a coalgebra with comultiplication $\bicomu$ and counit $\bicounit$ (all morphisms in $\scr{C}$).
	We call $H$ a \textbf{bialgebra}, if $\dbimu: \ol{H} \ot \ol{H} \rightarrow \ol{H}$ and $\dbiunit: \dcatunit \rightarrow \ol{H}$ are coalgebra morphisms in $\scr{D}$ (or equivalently if $\dbicomu: \ol{H} \rightarrow \ol{H} \ot \ol{H}$ and $\dbicounit: \ol{H} \rightarrow \dcatunit$ are algebra morphisms in $\scr{D}$). Furthermore $H$ is called a \textbf{Hopf algebra} if the identity $\id:H \rightarrow H$ is convolution invertible. The inverse $\coninv{\id} : H \rightarrow H$ is called \textbf{antipode of $H$}.
\end{defi}

\begin{rema}
	It is required that we assume a Hopf algebra $H$ to be an object in $\scr{C}$ and its structure morphisms $\bimu_H$, $\biunit_H$, $\bicomu_H$ and $\bicounit_H$ to be morphisms in $\scr{C}$, because we want to be able to pull these morphisms through the left braiding $\brd$ in both directions. For example for $X\in \scr{D}$ the relation
	\begin{align*}
		\brd_{H,X} (\ol{\bimu_H} \ot \id_X) = (\id_X \ot \ol{\bimu_H} ) \brd_{H\ot H,X}
	\end{align*}
	would not hold if we assume the multiplication $\ol{H} \ot \ol{H} \rightarrow \ol{H}$ in $\scr{D}$ is only a morphism in $\scr{D}$ without preimage of $\funCD$.
\end{rema}

\begin{defi}
	We define (left and right) \textbf{modules} and \textbf{comodules} over algebras and coalgebras, respectively, in a monoidal category $\scr{M}$ as usual, as they require no notion of a braiding. If $C \in \scr{M}$ is a coalgebra and $M \in \scr{M}$ a $C$-comodule, we denote its coaction with $\coact_M$.
\end{defi}

\begin{rema} \label{rem_tensor_mod_structure}
	We again have to carefully consider when tensor products of (co)modules inherit (co)module structure: Let $H \in \scr{C}$ be a bialgebra.
	\begin{itemize}
		\item If $M,N \in \scr{D}$ are left $\ol{H}$-modules with actions $\varphi_M, \varphi_N$, then $M \ot N$ is a left $\ol{H}$-module with action \begin{align*} (\varphi_M \ot \varphi_N) (\id \ot \brd_{H,M} \ot \id) (\ol{\bicomu_H} \ot \id \ot \id) \matdot \end{align*}
		\item If $M,N \in \scr{D}$ are right $\ol{H}$-comodules, then $M \ot N$ is a right $\ol{H}$-comodule with coaction \begin{align*}  (\id \ot \id \ot \ol{\bimu_H}) (\id \ot \brd_{H,N} \ot \id) (\coact_M \ot \coact_N) \matdot \end{align*}
	\end{itemize}
	Consider that a tensor product of right $\ol{H}$-modules is not naturally a right $\ol{H}$-module again and a tensor product of left $\ol{H}$-comodules is not a $\ol{H}$-comodule again, since we only have the braiding from one direction.
\end{rema}

\begin{defi}
	Suppose $H \in \scr{C}$ is a bialgebra. A \textbf{(left) $\ol{H}$-module algebra} is an algebra $A \in \scr{D}$ that is a left $\ol{H}$-module, such that its multiplication $\bimu_A : A\ot A \rightarrow A$ and unit $\biunit_A : \dcatunit \rightarrow A$ are left $\ol{H}$-module morphisms.
	
	A \textbf{(right) $\ol{H}$-comodule algebra} is an algebra $B \in \scr{D}$ that is a right $\ol{H}$-comodule, such that its multiplication $\bimu_B$ and unit $\biunit_B$ are right $\ol{H}$-comodule morphisms (or equivalently if the coaction $\coact_B: B \rightarrow B \ot \ol{H}$ is an algebra morphism).
\end{defi}

\begin{rema}
	A definition of a \textbf{right} $\ol{H}$-module algebra or a $\textbf{left}$ $\ol{H}$-comodule algebra, where $H \in \scr{C}$ is a bialgebra, is not possible in our context, see Remark \ref{rem_tensor_mod_structure}.
\end{rema}

\begin{defi}
	Suppose $H \in \scr{C}$ is a bialgebra. A \textbf{measuring on $\ol{H}$} is an algebra $A \in \scr{D}$ and a morphism $\mact_A: \ol{H} \ot A \rightarrow A$ such that
	\begin{align}
		\mact_A (\ol{\biunit_H} \ot \id_A) =& \id_A  \matcom
		\\ \mact_A (\id_{\ol{H}} \ot \biunit_A) =& \biunit_A \ol{\bicounit_H}  \matcom
		\\  \mact_A (\id_{\ol{H}} \ot \bimu_A) =& \bimu_A (\mact_A \ot \mact_A) (\id_{\ol{H}} \ot \brd_{H, A} \ot \id_A) (\ol{\bicomu_H} \ot \id_A^2) \label{defi_meas3}
		\matdot
	\end{align}
	If $A$ is a measuring on $\ol{H}$, then we always denote the measuring morphism with $\mact_A$. Moreover we denote
	\begin{align*}
	\brd^\nu_{H,A} :=& (\mact_A \otimes \id)(\id\otimes \brd_{H,A})(\dbicomu \otimes \id) : \ol{H} \otimes A \to A\otimes \ol{H} \matcom
	\end{align*}
	such that $(\ref{defi_meas3})$ can be written as
	\begin{align}
	\mact_A (\id_{\ol{H}} \ot \bimu_A) =& \bimu_A (\id_A \ot \mact_A) (\brd^\nu_{H,A} \ot \id_A)  \label{defi_meas3p} \matdot
	\end{align}
	
	Finally a morphism $f: A \rightarrow A'$ between two measurings $A,A'$ on $\ol{H}$ is called a measuring morphism, if it is an algebra morphism and $f \mact_A = \mact_{A'} (\id \ot f)$.
\end{defi}

\begin{exa}
	Let $H \in \scr{C}$ be a bialgebra. Trivially $\dcatunit$ is a measuring on $\ol{H}$, where $\mact_{\dcatunit} = \ol{\bicounit_H}$.
\end{exa}

\begin{rema}
	Let $H \in \scr{C}$ be a bialgebra. By definition an $\ol{H}$-module algebra is a measuring $A$ on $\ol{H}$ where the measuring morphism $\mact_A$ is associative, i.e. $\mact_A (\id_{\ol{H}} \ot \mact_A) = \mact_A ( \dbimu \ot \id_A)$.
\end{rema}

\begin{defi}
	Let $H \in \scr{C}$ be a bialgebra and let $M \in \scr{D}$ be a right $\ol{H}$-comodule. If $M_0\in \scr{D}$ and $\iota : M_0 \rightarrow M$ is a morphism, such that 
	\begin{align*}
	\begin{tikzcd}[ampersand replacement=\&]
	M_0 \arrow{r}{\iota} \& M \arrow[shift left]{r}{\coact_M} \arrow[',shift right]{r}{\id \ot \ol{\biunit_H}} \& M \ot \ol{H}
	\end{tikzcd} 
	\end{align*}
	is an equalizer diagram, then we call $M_0$ the \textbf{coinvariants of $M$} and denote $\coinv{M} := M_0$.
\end{defi}

\begin{exa} \label{exa_equalizer_alg1}
	Let $H \in \scr{C}$ be a bialgebra and let $B$ be an $\ol{H}$-comodule algebra. Then $\coinv{B}$ has an algebra structure: Let $\iota: \coinv{B} \rightarrow B$ be the equalizer morphism. 
	By definition we have
	\begin{align*}
		\coact_B \biunit_B &= (\id \ot \ol{\biunit_H}) \biunit_B \matcom \\
		\coact_B \bimu_B (\iota \ot \iota) &= (\id \ot \ol{\biunit_H}) \bimu_B (\iota \ot \iota)
	\end{align*}
	hence there exist unique morphisms $\biunit_{\coinv{B}} : \dcatunit \rightarrow \coinv{B}$ and $\bimu_{\coinv{B}} : \coinv{B} \ot \coinv{B} \rightarrow \coinv{B}$, such that $\iota \biunit_{\coinv{B}} = \biunit_B$ and $\iota \bimu_{\coinv{B}} = \bimu_B (\iota \ot \iota)$. Now
	\begin{align*}
		\iota \bimu_{\coinv{B}} ( \id \ot \biunit_{\coinv{B}}) = \bimu_B (\iota \ot \iota) ( \id \ot \biunit_{\coinv{B}}) = \bimu_B (\iota \ot \biunit_B) = \iota \, \id \matcom
	\end{align*}
	Hence $\bimu_{\coinv{B}} ( \id \ot \biunit_{\coinv{B}}) = \id$ by uniqueness of such morphisms. Similary $\bimu_{\coinv{B}} ( \biunit_{\coinv{B}} \ot \id ) = \id$. Finally
	\begin{align*}
		\iota \bimu_{\coinv{B}} ( \id \ot \bimu_{\coinv{B}}) =& \bimu_B (\id \ot \bimu_B) (\iota \ot \iota \ot \iota) \\=& \bimu_B ( \bimu_B \ot \id) (\iota \ot \iota \ot \iota) = \iota \bimu_{\coinv{B}} ( \bimu_{\coinv{B}} \ot \id)
		\matcom
	\end{align*}
	implying the associativity of $\bimu_{\coinv{B}}$.
\end{exa}

\section{Two-cocycles and crossed products}
Let $H \in \scr{C}$ be a bialgebra with multiplication $\bimu$, unit $\biunit$, comultiplication $\bicomu$ and counit $\bicounit$ (these will be the only structure morphisms where the subscript is omitted). Moreover let $A \in \scr{D}$ be a measuring on $\ol{H}$ and let $\sigma : \ol{H} \ot \ol{H} \rightarrow A$ be a morphism in $\scr{D}$.

\begin{defi} \label{def_cocycle}
	We define the morphisms $\hat{\sigma}: \ol{H} \ot \ol{H} \rightarrow A \ot \ol{H}$ and $\sigmul : (A \ot \ol{H}) \ot (A \ot \ol{H}) \rightarrow (A \ot \ol{H})$ as follows: 
	\begin{align*}
		\hat{\sigma} :=& (\sigma \otimes \dbimu)\bicomu_{\ol{H} \otimes \ol{H}} \matcom
		\\
		\sigmul :=& (\bimu_A \ot \id) (\bimu_A \ot \hat{\sigma}) (\id \ot \brd^\nu_{H,A} \ot \id) \matdot
	\end{align*}
	We call $\sigma : \ol{H} \ot \ol{H} \rightarrow A$ a \textbf{(two-)cocycle}, if it is convolution invertible and if the following relations hold:
	\begin{align}
		\bimu_A( \id \ot \sigma)(\hat{\sigma} \ot \id_{\ol{H}}) &= \bimu_A (\id \ot \sigma)(\brd^\mact_{H,A} \ot \id)(\id_{\ol{H}} \ot \hat{\sigma}) \label{defi_cocycle_rel1} \matcom 
		\\  \bimu_A(\id \ot \mact_A)(\hat{\sigma} \ot \id_A) &= \bimu_A (\id \ot \sigma)(\brd^\mact_{H,A} \ot \id)(\id_{\ol{H}} \otimes \brd^\mact_{H,A}) \label{defi_cocycle_rel2} \matcom 
		\\ \sigma (\dbiunit \ot \dbiunit) &= \biunit_A \label{defi_cocycle_rel3} \matdot 
	\end{align}
\end{defi}

\begin{rema} \label{rema_def_cocycle_rels}
	The following relations hold:
	\begin{align*}
		(\id \ot \dbicounit) \hat{\sigma} =& \sigma \matcom & 
		(\id \ot \dbicounit) \brd^\mact_{H,A} =& \mact_A \matcom &
		\brd^\mact_{H,A} (\id \ot \biunit_A) =& \biunit_A \ot \id \matdot 
	\end{align*}
\end{rema}

\begin{rema} \label{rema_cocycle_axiomtwo}
	Observe that $(\ref{defi_cocycle_rel2})$ is always true if $A=\dcatunit$, since then $\mact_A=\dbicounit$ and $\brd^\mact_{H,A} = \id_{\ol{H}}$. It can be viewed as a generalisation of the associativity of $\mact_A$: Indeed if $\sigma = \biunit_A (\dbicounit \ot \dbicounit)$, then $(\ref{defi_cocycle_rel1})$ and $(\ref{defi_cocycle_rel3})$ trivially hold and $(\ref{defi_cocycle_rel2})$ is precisely
	\begin{align*}
		\mact_A (\dbimu \ot \id) = \mact_A (\id \ot \mact_A) \matcom
	\end{align*}
	i.e. the associativity of $\mact_A$. This also implies that $\sigma = \biunit_A (\dbicounit \ot \dbicounit)$ is a cocycle if and only if $A$ is a $\ol{H}$-module algebra. In this case
	\begin{align*}
	\sigmul = (\bimu_A \ot \dbimu) (\id \ot \brd^\mact_{H,A} \ot \id) \matcom
	\end{align*}
	is an associative multiplication on $A \ot \ol{H}$ with unit $\biunit_A \ot \dbiunit$. This algebra structure is known as the \textbf{smash product $A\# \ol{H}$}.
\end{rema}

\begin{rema}
	An important thing to observe is that $\sigma$ and the given multiplication $\sigmul$ are morphisms in $\scr{D}$ and not in $\scr{C}$. Former considerations, like \cite{MR1765310} had a setting where there is only one braided monoidal category, i.e. $\scr{C}=\scr{D}$, $\funCD=\id$ and thus cocycles and multiplications always were required to be morphisms in that category.
\end{rema}

\begin{lem} \label{lem_cocycle_rels}
	If $\sigma : \ol{H} \ot \ol{H} \rightarrow A$ is a cocycle, then the following relations hold: 
	\begin{align}
	(\sigma \ot \dbicounit) \conv (\sigma (\dbimu \ot \id_{\ol{H}})) &= \left(\mact_A (\id_{\ol{H}} \ot \sigma)\right) \conv \left(\sigma (\id_{\ol{H}} \ot \dbimu) \right) \label{lem_cocycle_rels1} \\
	(\sigma (\dbimu \ot \id_{\ol{H}})) \conv \left(\coninv{\sigma} (\id_{\ol{H}} \ot \dbimu) \right) &= (\coninv{\sigma} \ot \dbicounit) \conv \left(\mact_A (\id_{\ol{H}} \ot \sigma)\right) \label{lem_cocycle_rels2}
	\\ \left(\sigma (\id_{\ol{H}} \ot \dbimu) \right) \conv (\coninv{\sigma} (\dbimu \ot \id_{\ol{H}})) &= \left(\mact_A (\id_{\ol{H}} \ot \coninv{\sigma})\right) \conv (\sigma \ot \dbicounit) \label{lem_cocycle_rels3}
	\\ \sigma (\dbiunit \ot \id) =\,& \biunit_A \dbicounit = \sigma(\id \ot \dbiunit) \label{lem_cocycle_rels5}
	\\ \coninv{\sigma} (\dbiunit \ot \id) =\,& \biunit_A \dbicounit = \coninv{\sigma} (\id \ot \dbiunit) \matdot \label{lem_cocycle_rels6}
	\end{align}
\end{lem}
\begin{proof}
	$(\ref{lem_cocycle_rels1})$ is precisely $(\ref{defi_cocycle_rel1})$. Now Lemma \ref{lem_conv_alg_morphism} gives
	\begin{align*}
		\coninv{(\sigma (\dbimu \ot \id))} &= (\coninv{\sigma} (\dbimu \ot \id)) & \coninv{(\sigma (\id \ot \dbimu))} &= (\coninv{\sigma} (\id \ot \dbimu))
		\\ \coninv{(\sigma (\dbiunit \ot \id))} &= \coninv{\sigma} (\dbiunit \ot \id) & \coninv{(\sigma (\id \ot \dbiunit))} &= \coninv{\sigma} (\id \ot \dbiunit)
	\end{align*}
	and it is elementary to check that
	\begin{align*}
		\coninv{(\sigma \ot \dbicounit)} &= \coninv{\sigma} \ot \dbicounit & \coninv{( \mact_A (\id \ot \sigma))} &=  \mact_A (\id \ot \coninv{\sigma} )  \matdot
	\end{align*}
	Hence $(\ref{lem_cocycle_rels2})$ and $(\ref{lem_cocycle_rels3})$ can be deduced from $(\ref{lem_cocycle_rels1})$, by multiplying this relation with the specific convolution inverse morphisms. The first relation of $(\ref{lem_cocycle_rels6})$ can be followed by applying $ \dbiunit \ot \dbiunit \ot \id $ from the right to $(\ref{lem_cocycle_rels3})$, the second by applying $\id \ot \dbiunit \ot \dbiunit$ from the right to $(\ref{lem_cocycle_rels2})$. Finally $(\ref{lem_cocycle_rels5})$ is implied by $(\ref{lem_cocycle_rels6})$.
\end{proof}

\begin{lem} \label{lem_sigmul_eps}
	The following relation holds:
	\begin{align*}
		(\id \ot \dbicounit) \sigmul (\biunit_A \ot \id \ot \biunit_A \ot \id) = \sigma \matdot
	\end{align*}
	In particular, if $\pi : \ol{H} \ot \ol{H} \rightarrow A$ is another morphism in $\scr{D}$, such that $\sigmul = \mu_\pi$, then $\sigma=\pi$.
\end{lem}
\begin{proof}
	Using Remark \ref{rema_def_cocycle_rels}, the calculation is straight forward:
	\begin{align*}
		& \sigmul (\biunit_A \ot \id \ot \biunit_A \ot \id) \\ = &(\bimu_A \ot \id) (\bimu_A \ot \hat{\sigma}) (\id \ot \brd^\nu_{H,A} \ot \id) (\biunit_A \ot \id \ot \biunit_A \ot \id)
		\\ =& (\bimu_A \ot \id) (\bimu_A \ot \hat{\sigma})  (\biunit_A \ot  \biunit_A \ot \id \ot  \id)
		 = \hat{\sigma}
	 	\matdot
	\end{align*}
	Now applying $\id \ot \dbicounit$ from the left gives $\sigma$.
\end{proof}

\begin{lem} \label{lem_sigmul_assoc}
	If $\sigma : \ol{H} \ot \ol{H} \rightarrow A$ is a cocycle, then the following relations hold:
	\begin{align}
		\brd^\mact_{H,A} (\id_{\ol{H}} \ot \bimu_A) =& (\bimu_A \ot \id_{\ol{H}}) (\id_A \ot \brd^\mact_{H,A}) (\brd^\mact_{H,A} \ot \id_A) \matcom \label{lem_sigmul_assoc1} \\
		\sigmul (\hat{\sigma} \ot \id_A \ot \id_{\ol{H}}) =& \sigmul (\brd^\mact_{H,A} \ot \hat{\sigma}) (\id_{\ol{H}} \ot \brd^\mact_{H,A} \ot \id_{\ol{H}}) \label{lem_sigmul_assoc2} \matdot
	\end{align}
\end{lem}
\begin{proof}
	It might be easier to follow the steps of this proof by using a notation similar to the one described in \cite{MR2948489}. $(\ref{lem_sigmul_assoc1})$ follows quickly from $(\ref{defi_meas3})$, the braid axioms and coassociativity of $\dbicomu$. 
	
	Now (\ref{defi_cocycle_rel1}), the bialgebra axiom $ \dbicomu \dbimu = (\dbimu \ot \dbimu) \bicomu_{\ol{H} \ot \ol{H}} $ as well as the coassociativity of $\bicomu_{\ol{H} \ot \ol{H}}$ imply
	\begin{align*}
		&(\bimu_A \ot \id)(\id \ot \hat{\sigma}) (\hat{\sigma} \ot \id_{\ol{H}}) 
		\\=& \left( \left( \bimu_A  (\id \ot \sigma ) (\hat{\sigma} \ot \id) \right) \ot \dbimu^2 \right) (\id^2 \ot \brd_{H \ot H,\ol{H}} \ot \id) (\bicomu_{\ol{H} \ot \ol{H}} \ot \dbicomu)
		\\=& \left( \left( \bimu_A (\id \ot \sigma)(\brd^\mact_{H,A} \ot \id)(\id_{\ol{H}} \ot \hat{\sigma}) \right) \ot \dbimu^2 \right)  (\id \ot \brd_{H,\ol{H\ot H}} \ot \id^2) \\& (\dbicomu \ot \bicomu_{\ol{H} \ot \ol{H}})
		\\=& (\bimu_A \ot \id)(\id \ot \hat{\sigma})(\brd^\mact_{H,A} \ot \id)(\id_{\ol{H}} \ot \hat{\sigma}) \matdot
	\end{align*}
	Similarly (\ref{defi_cocycle_rel2}) implies
	\begin{align*}
		 &(\bimu_A \ot \id)(\id \ot \brd^\mact_{H,A})(\hat{\sigma} \ot \id_A) 
		 \\ =& \left( \left( \bimu_A (\id \ot \sigma)(\brd^\mact_{H,A} \ot \id)(\id_{\ol{H}} \otimes \brd^\mact_{H,A}) \right) \ot \dbimu \right) \\& (\id^2 \ot \brd_{H\ot H, \ol{H}} \ot \id) (\bicomu_{\ol{H} \ot \ol{H}} \ot \dbicomu)
		 \\ =& (\bimu_A \ot \id) (\id \ot \hat{\sigma})(\brd^\mact_{H,A} \ot \id)(\id_{\ol{H}} \otimes \brd^\mact_{H,A}) \matcom
	\end{align*}
	Combining these yields
	\begin{align*}
		&\sigmul (\brd^\mact_{H,A} \ot \hat{\sigma}) (\id_{\ol{H}} \ot \brd^\mact_{H,A} \ot \id_{\ol{H}})
		\\=& (\bimu_A^2 \ot \id) (\id^2 \ot \hat{\sigma})(\id \ot \hat{\sigma} \ot \id) (\brd^\mact_{H,A} \ot \id^2)  (\id \ot \brd^\mact_{H,A} \ot \id )
		\\=& \sigmul (\hat{\sigma} \ot \id_A \ot \id_{\ol{H}})
		\matcom
	\end{align*}
	using also the associativity of $A$.
\end{proof}

\begin{prop} \label{prop_sigmul_assoc}
	Assume that $\sigma : \ol{H} \ot \ol{H} \rightarrow A$ is convolution invertible. Then $\sigma$ is a cocycle if and only if $\sigmul$ is an associative multiplication with unit $\biunit_A \ot \dbiunit$.
\end{prop}
\begin{proof}
	Assume that $\sigma$ is a cocycle. To show the associativity of $\sigmul$ we use $(\ref{lem_sigmul_assoc1})$ two times in the first equality and $(\ref{lem_sigmul_assoc2})$ in the second equality:
	\begin{align*}
		& \sigmul ( \id \ot \id \ot \sigmul) 
		\\=& (\bimu_A^2 \ot \id) \left( \id^2 \ot \left(  \sigmul (\brd^\mact_{H,A} \ot \hat{\sigma}) (\id_{\ol{H}} \ot \brd^\mact_{H,A} \ot \id_{\ol{H}}) \right) \right)
		\\& \left( \id \ot \brd^\mact_{H,A} \ot \id \ot \id \ot \id \right)
		\\=& (\bimu_A^2 \ot \id) \left( \id^2 \ot \left( \sigmul (\hat{\sigma} \ot \id_A \ot \id_{\ol{H}}) \right) \right)   \left( \id \ot \brd^\mact_{H,A} \ot \id \ot \id \ot \id \right)
		\\ =& \sigmul (\sigmul \ot \id \ot \id) 
		\matdot
	\end{align*}
	Now $(\ref{lem_cocycle_rels5})$ implies $\hat{\sigma} ( \id \ot \dbiunit) = \biunit_A \ot \id$. Thus using Remark \ref{rema_def_cocycle_rels} we obtain
	\begin{align*}
		&\sigmul (\id \ot \id  \ot \biunit_A \ot \dbiunit) 
		\\=& (\bimu_A \ot \id) (\bimu_A \ot \hat{\sigma}) (\id \ot \brd^\nu_{H,A} \ot \id) (\id \ot \id  \ot \biunit_A \ot \dbiunit) 
		\\=& (\bimu_A \ot \id) (\bimu_A \ot \hat{\sigma}) (\id \ot \biunit_A  \ot \id \ot \dbiunit) = \id \ot \id
	\end{align*}
	and similarly $\sigmul (\biunit_A \ot \dbiunit \ot \id \ot \id) = \id \ot \id$. Hence $\biunit_A \ot \dbiunit$ is the unit of $\sigmul$.
	
	Conversely assume that $\sigmul$ is an associative multiplication with unit $\biunit_A \ot \dbiunit$. We have
	\begin{align*}
		& \bimu_A( \id \ot \sigma)(\hat{\sigma} \ot \id_{\ol{H}})
		\\=& (\id \ot \dbicounit) \sigmul( \sigmul \ot \id \ot \id) (\biunit_A \ot \id \ot \biunit_A \ot \id \ot \biunit_A \ot \id)
		\\=& (\id \ot \dbicounit) \sigmul( \id \ot \id \ot \sigmul) (\biunit_A \ot \id \ot \biunit_A \ot \id \ot \biunit_A \ot \id)
		\\=& \bimu_A (\id \ot \sigma)(\brd^\mact_{H,A} \ot \id)(\id_{\ol{H}} \ot \hat{\sigma})
		\matcom
	\end{align*}
	and similarly
	\begin{align*}
		& \bimu_A(\id \ot \mact_A)(\hat{\sigma} \ot \id_A)
		\\=& (\id \ot \dbicounit) \sigmul( \sigmul \ot \id \ot \id) (\biunit_A \ot \id \ot \biunit_A \ot \id \ot \id \ot \dbiunit)
		\\=& (\id \ot \dbicounit) \sigmul( \id \ot \id \ot \sigmul) (\biunit_A \ot \id \ot \biunit_A \ot \id \ot \id \ot \dbiunit)
		\\=& \bimu_A (\id \ot \sigma)(\brd^\mact_{H,A} \ot \id)(\id_{\ol{H}} \otimes \brd^\mact_{H,A})
		\matdot
	\end{align*}
	Using Lemma \ref{lem_sigmul_eps} we obtain
	\begin{align*}
		\sigma (\dbiunit \ot \dbiunit) = (\id \ot \dbicounit) \sigmul (\biunit_A \ot \dbiunit \ot \biunit_A \ot \dbiunit)
		= (\id \ot \dbicounit) (\biunit_A \ot \dbiunit) = \biunit_A  \matcom
	\end{align*}
	hence $\sigma$ is a cocycle.
\end{proof}

\begin{defi}
	Let $\sigma : \ol{H} \ot \ol{H} \rightarrow A$ be a cocycle. By the previous proposition $A\ot \ol{H}$ becomes an algebra in $\scr{D}$ with multiplication $\sigmul$ and unit $\biunit_A \ot \dbiunit$. This algebra structure is called the \textbf{crossed product of $\sigma$}. We denote it with $\sigsmash{A}$.
\end{defi}

\begin{prop}
	Let $\sigma : \ol{H} \ot \ol{H} \rightarrow A$ be a cocycle. Then the algebra $\sigsmash{A}$ is an $\ol{H}$-comodule algebra with coaction $\id_A \ot \dbicomu: \sigsmash{A} \rightarrow \sigsmash{A} \ot \ol{H}$.
\end{prop}
\begin{proof}
	We have
	\begin{align*}
	&(\sigmul \ot \id_{\ol{H}}) \coact_{\sigsmash{A} \ot \sigsmash{A}} \\=& (\bimu_A \ot \id \ot \id)(\bimu_A \ot \sigma \ot \dbimu \ot \dbimu) (\id \ot \mact_A \ot \bicomu_{\ol{H} \ot \ol{H}}^2)  \\& (\id \ot \id \ot \brd_{H,A} \ot \id)(\id \ot \dbicomu \ot \id \ot \id)
	= (\id_A \ot \dbicomu) \sigmul 
	\end{align*}
	and trivially $(\id_A \ot \dbicomu) (\biunit_A \ot \dbiunit)  =  (\biunit_A \ot \dbiunit \ot \id_{\ol{H}}) \dbiunit$.
\end{proof}

\section{Cleft extensions}

Let $H \in \scr{C}$ be a Hopf algebra with multiplication $\bimu$, unit $\biunit$, comultiplication $\bicomu$, counit $\bicounit$ and with antipode $\antip$.

\begin{rema}
	The calculations throughout this section might be easier to follow by using a notation similar to the one described in \cite{MR2948489}.
\end{rema}

\begin{defi} \label{def_cleft}
	An $\ol{H}$-comodule algebra $B \in \scr{D}$ is called an \textbf{$\ol{H}$-cleft extension}, if there exists a convolution invertible morphism of comodules $\sect : \ol{H} \rightarrow B$. Such $\sect$ is called a \textbf{section}. $B$ is called an \textbf{$\ol{H}$-cleft object}, if $\coinv{B} \cong \dcatunit$, i.e. if
	\begin{align*}
	\begin{tikzcd}[ampersand replacement=\&]
	\dcatunit \arrow{r}{\biunit_B} \& B \arrow[shift left]{r}{\coact_B} \arrow[', shift right]{r}{\id \ot \dbiunit} \& B \ot \ol{H}
	\end{tikzcd} \matdot 
	\end{align*}
	is an equalizer diagram. 
	
	A morphism of two cleft extensions $B,B'$ is an $\ol{H}$-comodule algebra morphism $\alpha: B \rightarrow B'$. 
\end{defi}

\begin{rema} \label{rema_cleft_ext_morphism_unit}
	Let $B \in \scr{D}$ be an $\ol{H}$-cleft extension with section $\sect$. Then 
	\begin{align*}
		\sect' := \bimu_B (\coninv{\sect} \ot \sect) (\dbiunit \ot \id_{\ol{H}})
	\end{align*}
	 is a morphism of comodules with $\sect' \dbiunit = \biunit_B$ that is convolution invertible with inverse $\bimu_B (\coninv{\sect} \ot \sect) (\id_{\ol{H}} \ot \dbiunit)$.
	Hence we can without restriction assume that $\sect \dbiunit = \biunit_B$. We also obtain $\coninv{\sect} \dbiunit = \coninv{(\sect \dbiunit)} = \biunit_B$.
\end{rema}

\begin{prop} \label{prop_cocycle_yields_cleft_obj}
	Let $A \in \scr{D}$ be a measuring on $\ol{H}$ and $\sigma : \ol{H} \ot \ol{H} \rightarrow A$ be a cocycle. Then the $\ol{H}$-comodule algebra $\sigsmash{A}$ is an $\ol{H}$-cleft extension with section $\biunit_A \ot \id$ and coinvariants $A$. In particular if $A = \dcatunit$, then $\sigsmash{\dcatunit}$ is an $\ol{H}$-cleft object.
\end{prop}
\begin{proof}
	Clearly $\biunit_A \ot \id: \ol{H} \rightarrow \sigsmash{A}$ is a morphism of $\ol{H}$-comodules. We show that it is convolution invertible with inverse morphism 
	\begin{align*}
		\brd_{H,A} (\id \ot \coninv{\sigma}) (\dantip \ot \dantip \ot \id) \ol{\bicomu^2} \matdot
	\end{align*}
	The following calculation uses $(\ref{lem_cocycle_rels3})$ in the third last equality, $(\ref{lem_cocycle_rels5})$ and $(\ref{lem_cocycle_rels6})$ in the last equality as well as $\dbicomu S = \dbrd (S \ot S) \dbicomu$:
	\begin{align*}
	&(\biunit_A \ot \id) \conv \left( \brd_{H,A} (\id \ot \coninv{\sigma}) (\dantip \ot \dantip \ot \id) \ol{\bicomu^2} \right) \matdot
	\\=& (\bimu_A \ot \id) (\id \ot \sigma \ot \dbimu) (\mact_A \ot \bicomu_{\ol{H} \ot \ol{H}}) (\id \ot \brd_{H \ot H, A}) \\& (\id^3 \ot \coninv{\sigma}) (\id^2 \ot \dantip \ot \dantip \ot \id) \dbicomu^4
	\\=& (\bimu_A \ot \id)(\id \ot \sigma \ot \dbimu)(\mact_A \ot \id \ot \dbrd \ot \id) (\id \ot \brd_{H \ot H \ot H \ot H, A}) \\& (\id^3 \ot \dbrd \ot \coninv{\sigma}) (\id^3 \ot \dantip \ot \dantip \ot \dantip \ot \id) \dbicomu^6
	\\=& (\bimu_A \ot \dbiunit)(\mact_A \ot \sigma) (\id \ot \brd_{H \ot H, A}) (\id^3 \ot \coninv{\sigma}) (\id^2 \ot \dantip \ot \dantip \ot \id) \dbicomu^4
	\\=& (\bimu_A \ot \dbiunit)(\mact_A \ot \id)(\id \ot \coninv{\sigma} \ot \sigma) (\id \ot \brd_{H \ot H,\ol{H}\ot \ol{H}}) \\& (\id^2 \ot \dantip \ot \dantip \ot \id) \dbicomu^4
	\\=& (\id \ot \dbiunit) \left( \left(\mact_A (\id \ot \coninv{\sigma})\right) \conv (\sigma \ot \dbicounit) \right) (\id \ot S \ot \id) \dbicomu^2
	\\=& (\id \ot \dbiunit) \left( \left(\sigma (\id \ot \dbimu) \right) \conv (\coninv{\sigma} (\dbimu \ot \id)) \right) (\id \ot S \ot \id) \dbicomu^2
	\\=& (\bimu_A \ot \dbiunit) (\sigma \ot \coninv{\sigma}) (\id \ot \dbrd \ot \id) \\& (\id \ot \dbimu \ot \dbimu \ot \id) (\id^2 \ot \dantip \ot \dantip \ot \id^2) \dbicomu^5
	= (\biunit_A \ot \dbiunit) \dbicounit \matdot
	\end{align*}
	Moreover
	\begin{align*}
	& \left( \brd_{H,A} (\id \ot \coninv{\sigma}) (\dantip \ot \dantip \ot \id) \ol{\bicomu^2} \right) \conv (\biunit_A \ot \id) 
	\\=& (\bimu_A \ot \id) (\id \ot \sigma \ot \dbimu) (\id \ot \bicomu_{\ol{H} \ot \ol{H}}) \\& (\brd_{H,A} \ot \id) (\id \ot \coninv{\sigma} \ot \id) (\dantip \ot \dantip \ot \id^2) \dbicomu^3
	\\=& (\bimu_A \ot \id) (\coninv{\sigma} \ot \sigma \ot \id) (\id \ot \dbrd \ot \id \ot \dbimu)(\brd_{H,\ol{H} \ot \ol{H} \ot \ol{H} \ot \ol{H} } \ot \id)\\&(\id \ot \dbrd \ot \id^3) (\dantip \ot \dantip \ot \dantip \ot \id^3)\dbicomu^5
	\\=& ((\coninv{\sigma} \conv \sigma) \ot \dbimu)(\brd_{H, \ol{H} \ot \ol{H}} \ot \id)(\dantip \ot \dantip \ot \id^2) \dbicomu^3
	= (\biunit_A \ot \dbiunit) \dbicounit \matdot
	\end{align*}
	This implies that $\sigsmash{A}$ is an $\ol{H}$-cleft extension with section $\biunit_A \ot \id$. Finally
	\begin{align*}
	\begin{tikzcd}[ampersand replacement=\&]
	A \arrow{r}{\id \ot \dbiunit} \& \sigsmash{A} \arrow[shift left]{r}{\id \ot \dbicomu} \arrow[', shift right]{r}{\id \ot \id \ot \dbiunit} \& \sigsmash{A} \ot \ol{H}
	\end{tikzcd} \matdot 
	\end{align*}
	is an equalizer diagram: If $M \in \scr{D}$ is an object with morphism $\alpha : M \rightarrow \sigsmash{A}$, such that $(\id \ot \dbicomu) \alpha = (\id \ot \id \ot \dbiunit) \alpha$, then $\beta:=(\id \ot \dbicounit) \alpha: M \rightarrow A$ is the unique morphism with $\alpha = (\id \ot \dbiunit) \beta$.
\end{proof}

\begin{lem} \label{lem_conv_inv_of_modole_morphism}
	Let $B \in \scr{D}$ be an $\ol{H}$-cleft extension with section $\sect$. Then
	\begin{align*}
		\coact_B \coninv{\sect} = (\coninv{\sect} \ot \dantip) \dbrd \dbicomu \matdot
	\end{align*}
\end{lem}
\begin{proof}
	Since $\sect$ is a comodule morphism we have $\coact_B \sect = (\sect \ot \id) \dbicomu$. Observe that this morphism goes from the coalgebra $\ol{H}$ to the algebra $B \ot \ol{H}$. Now by Lemma \ref{lem_conv_alg_morphism} $\coact_B \coninv{\sect}$ is the convolution inverse of $\coact_B \sect$ and furthermore $(\coninv{\sect} \ot \dantip) \dbrd \dbicomu$ is the convolution inverse of $(\sect \ot \id) \dbicomu$. The claim follows.
\end{proof}

\begin{prop} \label{prop_equalizer_alg2}
	Let $B$ be an $\ol{H}$-comodule algebra. If $B$ is an $\ol{H}$-cleft extension, then $\coinv{B}$ becomes a measuring on $\ol{H}$.
\end{prop}
\begin{proof}
	In Example \ref{exa_equalizer_alg1} we saw that $\coinv{B}$ is an algebra. Let $\iota: \coinv{B} \rightarrow B$ be the equalizer morphism and let $\sect$ be a section for $B$. Using Lemma \ref{lem_conv_inv_of_modole_morphism} we have
	\begin{align*}
		& \coact_B \bimu_B^2 (\sect \ot \id \ot \coninv{\sect}) (\id \ot \brd_{H,B}) (\dbicomu \ot \iota)
		\\=& \bimu_{B \ot \ol{H}} (\bimu_B \ot \dbimu \ot \id \ot \id) (\id \ot \brd_{H,B} \ot \id \ot \id \ot \id) \\& (\coact_B \ot \coact_B \ot \coact_B) (\sect \ot \id \ot \coninv{\sect}) (\id \ot \brd_{H,B}) (\dbicomu \ot \iota)
		\\=& (\bimu_B^2 \ot \id) (\sect \ot \id \ot \coninv{\sect} \ot \id) (\id \ot \id \ot \dbrd) (\id \ot \brd_{H,B} \ot \id)
		\\& (\id \ot \dbimu \ot \id \ot \id) (\id^2 \ot S \ot \brd_{H,B}) (\ol{\bicomu^3} \ot \iota)
		\\=& (\id \ot \dbiunit) \bimu_B^2 (\sect \ot \id \ot \coninv{\sect}) (\id \ot \brd_{H,B}) (\dbicomu \ot \iota) \matcom
	\end{align*}
	hence there exists a unique morphism $\mact_{\coinv{B}} : \ol{H} \ot \coinv{B} \rightarrow \coinv{B}$, such that
	\begin{align*}
		\iota \, \mact_{\coinv{B}} = \bimu_B^2 (\sect \ot \id \ot \coninv{\sect}) (\id \ot \brd_{H,B}) (\dbicomu \ot \iota) \matdot
	\end{align*}
	Considering Remark \ref{rema_cleft_ext_morphism_unit} it is trivial that $\iota \mact_{\coinv{B}} (\dbiunit \ot \id) = \iota$ and hence by uniqueness of such morphisms $\mact_{\coinv{B}} (\dbiunit \ot \id) = \id_{\coinv{B}}$. Moreover
	\begin{align*}
		\iota \, \mact_{\coinv{B}} (\id \ot \biunit_{\coinv{B}}) 
		=& \bimu_B^2 (\sect \ot \id \ot \coninv{\sect}) (\id \ot \brd_{H,B}) (\dbicomu \ot \biunit_B) 
		\\ =& \bimu_B (\sect \ot \coninv{\sect}) \dbicomu = \biunit_B \dbicounit
		 = \iota \, \biunit_{\coinv{B}} \dbicounit \matcom
	\end{align*}
	hence by uniqueness $\mact_{\coinv{B}} (\id \ot \biunit_{\coinv{B}}) = \biunit_{\coinv{B}} \dbicounit$. Finally
	\begin{align*}
		&\iota \, \bimu_{\coinv{B}}(\mact_{\coinv{B}} \ot \mact_{\coinv{B}}) (\id \ot \brd_{H,\coinv{B}} \ot \id) (\dbicomu \ot \id^2)
		\\=& \bimu_{B} (\iota \ot \iota) (\mact_{\coinv{B}} \ot \mact_{\coinv{B}}) (\id \ot \brd_{H,\coinv{B}} \ot \id) (\dbicomu \ot \id^2)
		\\=& \bimu_B^5 (\sect \ot \id \ot \coninv{\sect} \ot \sect \ot \id \ot \coninv{\sect}) (\id \ot \brd_{H,B} \ot \id \ot \id \ot \id) \\& (\id^2 \ot \brd_{H,B} \ot \brd_{H,B}) (\id^3 \ot \brd_{H,B} \ot \id) (\ol{\bicomu^3} \ot \iota \ot \iota) 
		\\=& \bimu_B^3 (\sect \ot \id^2 \ot \coninv{\sect}) (\id \ot \id^2 \ot \brd_{H,B} ) (\id \ot \brd_{H,B} \ot \id) (\dbicomu \ot \iota \ot \iota) 
		\\=& \bimu_B^2 (\sect \ot \id \ot \coninv{\sect}) (\id \ot \brd_{H,B}) (\dbicomu \ot \bimu_{B}) (\id \ot \iota \ot \iota)
		\\=& \bimu_B^2 (\sect \ot \id \ot \coninv{\sect}) (\id \ot \brd_{H,B}) (\dbicomu \ot \iota) (\id \ot \bimu_{\coinv{B}})
		\\=&\iota \, \mact_{\coinv{B}} (\id \ot \bimu_{\coinv{B}})
		  \matdot
	\end{align*}
	Hence $\coinv{B}$ is a measuring on $\ol{H}$.
\end{proof}

\begin{prop} \label{prop_cleft_obj_yields_cocycle}
	Let $B \in \scr{D}$ be an $\ol{H}$-comodule algebra. If $B$ is an $\ol{H}$-cleft extension with section $\sect: \ol{H} \rightarrow B$, then for
	\begin{align*}
		\tilde{\sigma} := (\bimu_B (\sect \ot \sect)) \conv (\coninv{\sect} \, \dbimu) : \ol{H} \ot \ol{H} \rightarrow B
	\end{align*}
	there exists a unique morphism $\sigma : \ol{H}  \ot \ol{H} \rightarrow \coinv{B}$, such that $\tilde{\sigma} = \iota \sigma$, where $\iota: \coinv{B} \rightarrow B$ is the equalizer morphism. 
	This $\sigma$ is a cocycle and $\bimu_B (\iota \ot \sect): \sigsmash{\coinv{B}} \rightarrow B$ is an isomorphism of $\ol{H}$-cleft extensions.
\end{prop}

\begin{rema*}
	Observe that $\bimu_B (\iota \ot \sect) (\biunit_{\coinv{B}} \ot \id) = \sect$, i.e. the isomorphism $ \bimu_B (\iota \ot \sect)$ is compatible with the underlying sections as well.
\end{rema*}

\begin{proof}
	Refer to Examples \ref{exa_equalizer_alg1} and Proposition \ref{prop_equalizer_alg2} to recall the algebra and measuring structure of $\coinv{B}$ and the relations between the structure morphisms and $\iota$.
	We first verify that $\coact_B \tilde{\sigma} =  (\id \ot \dbiunit) \tilde{\sigma}$: Using that $\sect$ is a comodule morphism, Lemma \ref{lem_conv_inv_of_modole_morphism} and $ \coact_B \bimu_B = (\bimu_B \ot \dbimu) (\id \ot \brd_{H,B} \ot \id) (\coact_B \ot \coact_B)$ we obtain:
	\begin{align*}
		\coact_B \tilde{\sigma} =& (\bimu_B \ot \dbimu) (\id \ot \brd_{H,B} \ot \id) (\coact_B \ot \coact_B) (\bimu_B \ot \coninv{\sect}) (\sect \ot \sect \ot \dbimu) \bicomu_{\ol{H} \ot \ol{H}}
		\\=&  (\bimu_B \ot \dbimu) (\id \ot \brd_{H \ot H, B}) (\bimu_B \ot \id \ot \dantip \ot \coninv{\sect})\\& (\sect \ot \sect \ot \dbimu \ot \dbimu \ot \dbimu) \bicomu_{\ol{H} \ot \ol{H}}^3
		\\=& (\bimu_B \ot \dbiunit) (\bimu_B \ot \coninv{\sect})(\sect \ot \sect \ot \dbimu) \bicomu_{\ol{H} \ot \ol{H}}
		= (\id \ot \dbiunit) \tilde{\sigma}
		\matdot
	\end{align*}
	Hence we obtain a unique morphism $\sigma : \ol{H}  \ot \ol{H} \rightarrow \coinv{B}$, such that $\tilde{\sigma} = \iota \sigma$. Letting $\tilde{\pi} := (\sect \dbimu) \conv (\bimu_B (\coninv{\sect} \ot \coninv{\sect}) \dbrd )$, using also $\dbimu (\dantip \ot \dantip) \dbrd = \dantip \dbimu $ we similarly to above obtain
	\begin{align*}
		\coact_B \tilde{\pi} =& (\bimu_B \ot \dbimu) (\id \ot \brd_{H \ot H, B})(\sect \ot \id \ot \dantip \ot \bimu_B(\coninv{\sect} \ot \coninv{\sect})) \\& (\dbimu \ot \dbimu \ot \dbimu \ot \dbrd) \bicomu_{\ol{H} \ot \ol{H}}^3 
		\\=& (\bimu_B \ot \dbiunit) ( \sect\dbimu \ot ( \bimu_B (\coninv{\sect} \ot \coninv{\sect}) \dbrd) ) \bicomu_{\ol{H} \ot \ol{H}}
		= (\id \ot \dbiunit) \tilde{\pi}
		\matcom
	\end{align*}
	hence there must exist a unique morphism $\pi : \ol{H}  \ot \ol{H} \rightarrow \coinv{B}$, such that $\tilde{\pi} = \iota \pi$. We show that $\sigma$ and $\pi$ are convolution inverse: First observe that $\tilde{\sigma}$ and $\tilde{\pi}$ are convolution inverse, since $\coninv{(\sect \dbimu)} = \coninv{\sect} \dbimu$ and $\coninv{(\bimu_B (\sect \ot \sect))} = \bimu_B (\coninv{\sect} \ot \coninv{\sect}) \dbrd$. Then
	\begin{align*}
		\iota (\sigma \conv \pi) = \bimu_B (\iota \sigma \ot \iota \pi) \bicomu_{\ol{H} \ot \ol{H}} = \tilde{\sigma} \conv \tilde{\pi} = \biunit_B (\dbicounit \ot \dbicounit) = \iota \biunit_{\coinv{B}} (\dbicounit \ot \dbicounit) \matcom
	\end{align*}
	hence $\sigma \conv \pi = \biunit_{\coinv{B}} (\dbicounit \ot \dbicounit)$ and similarly $\pi \conv \sigma = \biunit_{\coinv{B}} (\dbicounit \ot \dbicounit)$. 
	
	Now $f := \bimu_B (\iota \ot \sect)$ is an isomorphism: 
	Consider that by Lemma \ref{lem_conv_inv_of_modole_morphism} we have
	\begin{align*}
	&\coact_B \bimu_B (\id \ot \coninv{\sect}) \coact_B \\=& (\bimu_B \ot \dbimu) (\id \ot \brd_{H \ot H, B}) (\id \ot \id \ot \dantip \ot \coninv{\sect}) (\id \ot \dbicomu^2) \coact_B \\=& (\id \ot \dbiunit) \bimu_B (\id \ot \coninv{\sect}) \coact_B
	\end{align*}
	hence there exists a unique morphism $\alpha : B \rightarrow \coinv{B}$, such that $\iota \alpha = \bimu_B (\id \ot \coninv{\sect}) \coact_B$. Then setting $g := (\alpha \ot \id) \coact_B : B \rightarrow \sigsmash{\coinv{B}}$ we obtain $fg = \id_B$ and
	\begin{align*}
	\iota \alpha f = \bimu_B^2 (\iota \ot \sect \ot \coninv{\sect}) (\id \ot \dbicomu) = \iota (\id \ot \dbicounit) \matcom
	\end{align*}
	which implies $\alpha f = \id \ot \dbicounit$ and thus
	\begin{align*}
	gf &= (\alpha \ot \id) \coact_B \bimu_B (\iota \ot \sect)
	= (\alpha \ot \id) (f \ot \id) (\id \ot \dbicomu) = \id_{\coinv{B}} \ot \id_{\ol{H}} \matdot
	\end{align*}
	Since
	\begin{align*}
		f \sigmul =& \bimu_B^3 (\id \ot \id \ot \iota \sigma \ot \sect \dbimu) (\id \ot (\iota \mact_{\coinv{B}}) \ot \bicomu_{\ol{H} \ot \ol{H}}) \\& (\id \ot \id \ot \brd_{H,{\coinv{B}}} \ot \id) (\iota \ot \dbicomu \ot \id \ot \id)
		\\=& \bimu_B^7 (\id^6 \ot \coninv{\sect} \ot \sect) (\id^4 \ot \sect \ot \sect \ot \dbimu \ot \dbimu) (\id^3 \ot \coninv{\sect} \ot \bicomu_{\ol{H} \ot \ol{H}}^2) \\& (\id \ot \sect \ot \brd_{H \ot H,B} \ot \id) (\iota \ot \dbicomu^2 \ot \iota \ot \id) \\ =& \bimu_B^3 (\iota \ot \sect \ot \iota \ot \sect) = \bimu_B (f \ot f)
	\end{align*}
	and $f (\biunit_{\coinv{B}} \ot \dbiunit) = \biunit_B$ (see Remark \ref{rema_cleft_ext_morphism_unit}) we obtain that $\sigmul$ is an associative multiplication with unit $\biunit_{\coinv{B}} \ot \dbiunit$ and $f$ is an isomorphism of algebras. Proposition \ref{prop_sigmul_assoc} then implies that $\sigma$ is a cocycle.
	Finally since
	\begin{align*}
		\coact_B f = (\bimu_B \ot \dbimu) (\id \ot \brd_{H,B} \ot \id) (\coact_B \iota \ot \coact_B \sect) = (f \ot \id) (\id \ot \dbicomu)
	\end{align*}
	$f$ is a morphism of $\ol{H}$-comodules and thus an isomorphism of $\ol{H}$-comodule algebras.
\end{proof}

\begin{defi}
	Given an $\ol{H}$-cleft extension $B \in \scr{D}$ with section $\sect$, we denote the cocycle $\ol{H} \ot \ol{H} \rightarrow \coinv{B}$ from Proposition \ref{prop_cleft_obj_yields_cocycle} with $\sigma_{\sect}$.
\end{defi}

\begin{rema} \label{rema_cleft_obj_smash}
	Let $B$ be an $\ol{H}$-cleft extension with section $\sect$. If $\sect$ is an algebra morphism, then $\sigma_{\sect} = \biunit_{\coinv{B}} (\dbicounit \ot \dbicounit)$, hence $\coinv{B}$ is an $\ol{H}$-module algebra by $(\ref{defi_cocycle_rel2})$ and thus $B$ as a comodule algebra is isomorphic to the smash product $\coinv{B} \# \ol{H}$ (see Remark \ref{rema_cocycle_axiomtwo}).
\end{rema}

\begin{lem} \label{lem_cocycle_of_sect_of_cocycle}
	Let $A \in \scr{D}$ be a measuring on $\ol{H}$ and $\sigma : \ol{H} \ot \ol{H} \rightarrow A$ be a cocycle. Then $\sigma = \sigma_{\biunit_A \ot \id}$.
\end{lem}
\begin{proof}
	We must show
	\begin{align*}
	\sigma = (\id \ot \dbicounit) \left( \left( \sigmul (\sect \ot \sect) \right) \conv \left( \coninv{\sect} \dbimu \right) \right) \matcom
	\end{align*}
	where $\sect = \biunit_A \ot \id$ and $\coninv{\sect} = \brd_{H,A} (\id \ot \coninv{\sigma}) (\dantip \ot \dantip \ot \id) \ol{\bicomu^2}$ (see proof of Proposition \ref{prop_cocycle_yields_cleft_obj}). This is achieved as follows: 
	\begin{align*}
	&(\id \ot \dbicounit) \left( \left( \sigmul (\sect \ot \sect) \right) \conv \left( \coninv{\sect} \dbimu \right) \right)
	\\=& (\bimu_A \ot \dbicounit) (\bimu_A \ot \sigma \ot \dbimu) (\id \ot \mact_A \ot \bicomu_{\ol{H} \ot \ol{H}}) (\id \ot \id \ot \brd_{H\ot H,A}) \\& (\id \ot \id^3 \ot \coninv{\sigma}) (\id \ot \id^2 \ot \dantip \ot \dantip \ot \id) (\sigma \ot \dbimu \ot \dbimu \ot \dbimu \ot \dbimu \ot \dbimu) \bicomu_{\ol{H} \ot \ol{H}}^5
	\\=& (\bimu_A \ot \dbicounit) (\bimu_A \ot \sigma \ot \dbimu) (\id \ot \mact_A \ot \bicomu_{\ol{H} \ot \ol{H}}) (\id \ot \id \ot \brd_{H\ot H,A}) \\& (\id \ot \id^3 \ot \coninv{\sigma}) (\id \ot \id^2 \ot \dantip \ot \dantip \ot \id) (\id \ot \dbicomu^4) (\sigma \ot \dbimu) \bicomu_{\ol{H} \ot \ol{H}}
	\\=& (\id \ot \left( \sect \conv \coninv{\sect} \right) ) (\sigma \ot \dbimu) \bicomu_{\ol{H} \ot \ol{H}} = \sigma \matdot
	\end{align*}
	The calculation uses $\dbicomu \dbimu = (\dbimu \ot \dbimu) \bicomu_{\ol{H} \ot \ol{H}}$ to benefit from the coassociativity of $\bicomu_{\ol{H} \ot \ol{H}}$.
\end{proof}

We can give an equivalence of categories: 
\begin{itemize}
	\item Let $\scr{A}$ be the category where the objects are pairs $(A,\sigma)$, where $A$ is a measuring on $\ol{H}$ and $\sigma : \ol{H} \ot \ol{H} \rightarrow A$ is a cocycle and where a morphism $f : (A,\sigma) \rightarrow (A',\sigma')$ is a measuring morphism $f : A \rightarrow A'$, such that $ f \sigma = \sigma' $.
	\item Let $\scr{B}$ be the category where the objects are pairs $(B,\sect)$ of an $\ol{H}$-cleft extension $B$ and a section $\sect : \ol{H} \rightarrow B$, and where a morphism $f : (B,\sect) \rightarrow (B',\sect')$ is an $\ol{H}$-comodule algebra morphism $f : B \rightarrow B'$, such that $f \sect = \sect'$.
\end{itemize}

\begin{thm} \label{thm_cleft_cocycle}
	$\scr{A}$ and $\scr{B}$ are equivalent categories. An equivalence is given by the following functors:
	\begin{itemize}
		\item Let $F: \scr{A} \rightarrow \scr{B}$ be the functor that maps an object $(A,\sigma) \in \scr{A}$ to $(\sigsmash{A},\biunit_A \ot \id) \in \scr{B}$ and a morphism $f : (A,\sigma) \rightarrow (A',\sigma')$ in $\scr{A}$ to the morphism $f \ot \id : (\sigsmash{A},\biunit_A \ot \id) \rightarrow (\sigsmashp{A'}{\sigma'},\biunit_{A'} \ot \id) $ in $\scr{B}$.
		\item Let $G: \scr{B} \rightarrow \scr{A}$ the functor that maps an object $(B,\sect) \in \scr{B}$ to $(\coinv{B},\sigma_{\sect}) \in \scr{A}$ and if $f : (B,\sect) \rightarrow (B',\sect')$ is a morphism in $\scr{B}$ then let $G(f) : (\coinv{B},\sigma_{\sect}) \rightarrow (\coinv{B'},\sigma_{\sect'})$ be the unique morphism $\coinv{B} \rightarrow \coinv{B'}$ such that $\iota_{B'} G(f) = f \iota_B$, where $\iota_B : \coinv{B} \rightarrow B$, $\iota_{B'} : \coinv{B'} \rightarrow B'$ are the equalizer morphisms.
	\end{itemize}
\end{thm}
\begin{proof}
	The calculations to check that that $F$ is well defined are straight forward. For the calculations to check that $G$ is well defined observe that if $f : (B,\sect) \rightarrow (B',\sect')$ is a morphism in $\scr{B}$, then $\coninv{(\sect')} = \coninv{(f \sect)} = f \coninv{\sect}$, since $f$ is an algebra morphism. If $f : (A,\sigma) \rightarrow (A',\sigma')$ is a morphism in $\scr{A}$, then again $f$ is the unique morphism for $G(f\ot \id)$ such that 
	\begin{align*}
		(\id \ot \dbiunit) G(f\ot \id) = (f \ot \id) (\id \ot \dbiunit) \matdot
	\end{align*}
	Hence Lemma \ref{lem_cocycle_of_sect_of_cocycle} implies that $GF = \id_{\scr{A}}$.
	Moreover if $f : (B,\sect) \rightarrow (B',\sect')$ is a morphism in $\scr{B}$, then by Proposition \ref{prop_cleft_obj_yields_cocycle} $\bimu_B (\iota_B \ot \sect)$ and $\bimu_{B'} (\iota_{B'} \ot \sect')$ are isomorphisms in $\scr{B}$ and the diagram
	\begin{align*}\begin{tikzcd}[ampersand replacement=\&]
	\sigsmashp{\coinv{B}}{\sigma_{\sect}} \arrow{d}{\bimu_B (\iota_B \ot \sect)} \arrow{r}{G(f) \ot \id} \& \sigsmashp{\coinv{B'}}{\sigma_{\sect'}} \arrow{d}{\bimu_{B'} (\iota_{B'} \ot \sect')} \\
	B \arrow{r}{f} \& B'
	\end{tikzcd}\end{align*}
	commutes, since $f \sect = \sect'$. Hence $FG$ is naturally equivalent to $\id_{\scr{B}}$.
\end{proof}

Let $\scr{B}_\mathrm{ob}$ be the subcategory of $\scr{B}$ of objects $(B,\sect) \in \scr{B}$, such that $B$ is an $\ol{H}$-cleft object.

\begin{cor} \label{cor_cleft_cocycle}
	The following describe inverse maps between the set of cocycles $\sigma:\ol{H} \ot \ol{H} \rightarrow \dcatunit$ and the set of isomorphism classes of objects $(B,\sect) \in \scr{B}_\mathrm{ob}$, respectively:
	\begin{enumerate}
		\item Map a cocycle $\sigma : \ol{H} \ot \ol{H} \rightarrow \dcatunit$ to the isomorphism class of $(\sigsmash{\dcatunit},\id_{\ol{H}}) \in \scr{B}_\mathrm{ob}$. 
		\item Map the class of any $(B,\sect) \in \scr{B}_\mathrm{ob}$ to the cocycle $\sigma_{\sect}: \ol{H} \ot \ol{H} \rightarrow \dcatunit$.
	\end{enumerate}
\end{cor}

\section{Liftings of a coradically graded Hopf algebra} \label{sect_liftings}
	Let $\fK$ be a field, $H$ a Hopf algebra over $\fK$, $\scr{C} = \prescript{H}{H}{\scr{YD}}$ be the category of left-left Yetter-Drinfeld modules, $\scr{D}=H\mathrm{-mod}$ the category of left $H$-modules and $\funCD : \prescript{H}{H}{\scr{YD}} \rightarrow H\mathrm{-mod}$ the forgetful functor. Let $\brd$ be the left braiding on $\funCD$ as described in Example \ref{exa_yetter_drinfeld}. Finally let $R \in \prescript{H}{H}{\scr{YD}}$ be a Hopf algebra (in terms of Definition \ref{def_hopf}). Then $R$ is an $H$-module algebra (the structure morphisms of $R$ are morphisms in $\prescript{H}{H}{\scr{YD}}$) and the smash product $R \# H$ becomes a (classical) Hopf algebra with multiplication and comultiplication given by
	\begin{align*}
	\bimu_{R\# H}  &\left(r \ot h , r' \ot h' \right) = r (h_{(1)} \cdot r') \ot h_{(2)} h'
	\\ \bicomu_{R\# H} & (r \ot h) = \left( r_{(1)} \ot r_{(2)(-1)} h_{(1)} \right) \ot \left( r_{(2)(0)} \ot h_{(2)} \right)  \matcom
	\end{align*}
	where $r,r' \in R$, $h,h' \in H$, and with antipode 
	\begin{align*}
	\antip_{R\# H}  &= (\biunit_R \bicounit_R \ot \antip_H) \conv (\antip_R \ot \biunit_H \bicounit_H) \matdot
	\end{align*}
	We denote $\scr{H} := R \# H$.
	Let $\cocyclesetR$ denote the set of all two-cocycles $\pi : R \ot R \rightarrow \fK$ in $\scr{D}=H\mathrm{-mod}$, in terms of Definition~\ref{def_cocycle} (this means $\pi$ is an $H$-module morphism and the defining relations rely on the left braiding $\brd$).
	Moreover let $\cocyclesetH$ denote the set of all (classical) two-cocycles $\sigma : \scr{H} \ot \scr{H} \rightarrow \fK$ and let $\cocyclesetHp$ be the set of all $\sigma \in \cocyclesetH$, such that
	\begin{align} \label{rel_defining_zp}
	\sigma(r \ot h, r' \ot h') &= \sigma(r \ot 1, h \cdot r' \ot 1) \bicounit_H(h') \matcom 
	\end{align}
	for all $r, r'\in R, h,h' \in H$.
	
	Let $\cleft(R)$ denote the set of all isomorphism classes $[(\scr{E},\sect)]$ of pairs $(\scr{E},\sect)$, where $R$ is an $R$-cleft object in $\scr{D}=H\mathrm{-mod}$, in terms of Definition~\ref{def_cleft}, and $\sect: R \rightarrow \scr{E}$ a section (isomorphy as in the category $\scr{B}$ described in Theorem \ref{thm_cleft_cocycle}). 
	Similarly let $\cleft(\scr{H})$ denote the set of all isomorphism classes of pairs $(E,\sect)$, where $E$ is an $\scr{H}$-cleft object and $\sect: \scr{H} \rightarrow E$ a section. Finally let $\cleft'(\scr{H})$ be the set of all $[(E,\sect)] \in \cleft(\scr{H})$, such that 
	\begin{align} \label{rel_defining_cleftp}
		\sect ( (1 \ot h) (r \ot h') (1 \ot h'') ) &= \sect (1 \ot h) \sect (r \ot h') \sect (1 \ot h'') \matcom
	\end{align}
	for all $r \in R, h,h',h'' \in H$ (this property is indepentent of the representative, since if $(\ref{rel_defining_cleftp})$ holds for the pair $(E,\sect)$, then it also holds for all $(E',\sect') \in [(E,\sect)]$).
	 
	
	Finally let $\alpha_R: \cocyclesetR \rightarrow \cleft(R)$, $\alpha_{\scr{H}}: \cocyclesetH \rightarrow \cleft(\scr{H})$ be the bijections of Corollary \ref{cor_cleft_cocycle}, respectively.

	\begin{lem} \label{lem_cocyclep_rel}
		Let $\sigma \in \cocyclesetHp$. Then for $r,r' \in R, h,h' \in H$:
		\begin{align*}
		\sigma(r \ot h, r' \ot h') &= \sigma (r \ot h, r' \ot 1) \bicounit_{H} (h') \matcom
		\\ \sigma(r \ot h, 1 \ot h') &= \bicounit_R (r) \bicounit_H(h) \bicounit_{H} (h') \matcom
		\\ \sigma(1 \ot h, r \ot h') &= \bicounit_R (r) \bicounit_H(h) \bicounit_{H} (h') \matcom
		\\ \coninv{\sigma}(r \ot h, r' \ot h') &= \coninv{\sigma}(r \ot 1, h \cdot r' \ot 1) \bicounit_H(h')
		\matdot
		\end{align*}
	\end{lem}
	\begin{proof}
		The first three relations follow from (\ref{rel_defining_zp}). For the last relation view the right hand side as a map $\tilde{\sigma} : \scr{H} \ot  \scr{H} \rightarrow \fK$ and show that $\sigma \conv \tilde{\sigma} = \bicounit_{\scr{H} \ot \scr{H}}$, using that $\sigma \in \cocyclesetHp$, which implies $\tilde{\sigma} = \coninv{\sigma} \conv \sigma \conv \tilde{\sigma} = \coninv{\sigma}$.
	\end{proof}

	\begin{lem} \label{lem_cleftp_section_rel}
		Let $[(E,\sect)] \in \cleft'(\scr{H})$. Then for $r \in R, h,h',h'' \in H$:
		\begin{align*}
		\sect (r \ot hh') &= \sect (r \ot h)  \sect (1 \ot h') \matcom
		\\ \sect (1 \ot S_H(h) ) &= \coninv{\sect} (1 \ot h) \matcom 
		\\ \coninv{\sect} ( (1 \ot h) (r \ot h') (1 \ot h'') ) &= \coninv{\sect} (1 \ot h'') \coninv{\sect} (r \ot h') \coninv{\sect} (1 \ot h)
		\matdot
		\end{align*}
	\end{lem}
	\begin{proof}
		The first relation follows from (\ref{rel_defining_cleftp}), since $r \ot hh' = (r \ot h)(1 \ot h')$ in $\scr{H}$.
		For the second relation we simply have to verify that \begin{align*} \sect (\biunit_R \ot S_H) = \coninv{(\sect(\biunit_R \ot \id))} \matdot \end{align*} The third relation can be deduced from (\ref{rel_defining_cleftp}) by viewing it as two morphisms $H \ot \scr{H} \ot H \rightarrow E$ and calculating the convolution inverse on both sides.
	\end{proof}
	
	\begin{prop} \label{prop_cleftbij_restrictsp}
		$\alpha_{\scr{H}}$ restricts to a bijection $\cocyclesetHp \rightarrow \cleft'(\scr{H})$.
	\end{prop}
	\begin{proof}
		Let $\sigma \in \cocyclesetHp$, then one must show (\ref{rel_defining_cleftp}) for the trivial section $\sect=\id : \scr{H} \rightarrow \fK \#_\sigma \scr{H}$, which is straight forward. Conversely if $[(E,\sect)] \in \cleft'(\scr{H})$, then indeed (\ref{rel_defining_zp}) holds for $\sigma_{\sect}$:
		\begin{align*}
			&\sigma_{\sect}(r \ot h, r' \ot h') 
			\\ =& \sect \left( r_{(1)} \ot r_{(2)(-1)} h_{(1)} \right) \sect \left( {r'}_{(1)} \ot {r'}_{(2)(-1)} {h'}_{(1)} \right) 
			\\ &\coninv{\sect} \left( r_{(2)(0)} (h_{(2)} \cdot {r'}_{(2)(0)} ) \ot h_{(3)} {h'}_{(2)} \right)
			\\ =&  \sect \left( r_{(1)} \ot r_{(2)(-1)} \right) \sect \left(1 \ot h_{(1)} \right) \sect \left( {r'}_{(1)} \ot {r'}_{(2)(-1)} \right) 
			\\ &\coninv{\sect} \left( 1 \ot h_{(3)} \right) \coninv{\sect} \left( r_{(2)(0)} (h_{(2)} \cdot {r'}_{(2)(0)} ) \ot 1 \right) \bicounit_H (h')
			\\ =& \sect (r_{(1)} \ot r_{(2)(-1)}) \sect \left( h_{(1)} \cdot {r'}_{(1)} \ot h_{(2)} {r'}_{(2)(-1)} S_H (h_{(4)} ) \right) \\& \coninv{\sect} \left( r_{(2)(0)} (h_{(3)} \cdot {r'}_{(2)(0)}) \ot 1 \right) \bicounit_H (h') 
			\\ =& \sect (r_{(1)} \ot r_{(2)(-1)}) \sect \left( h_{(1)} \cdot {r'}_{(1)} \ot (h_{(2)} \cdot {r'}_{(2)})_{(-1)} \right) \\& \coninv{\sect} \left( r_{(2)(0)} (h_{(2)} \cdot {r'}_{(2)})_{(0)} \ot 1 \right) \bicounit_H (h')
			\\ =& \sigma_{\sect}(r \ot 1, h \cdot r' \ot 1) \bicounit_H(h') \matdot
		\end{align*}
		The second and third equalities use Lemma \ref{lem_cleftp_section_rel} and the second last equality uses that $R$ is a Yetter-Drinfeld module.
	\end{proof}
	
	\begin{rema}
		Observe that an $R$-cleft object $\scr{E}$ is an $H$-module algebra, since $\bimu_{\scr{E}}$, $\biunit_{\scr{E}}$ are morphisms in $\scr{D}=H\mathrm{-mod}$.
	\end{rema}
	
	\begin{lem} \label{lem_cleftr_gives_clefth}
		Let $[(\scr{E},\sect_\scr{E})] \in \cleft(R)$. Then $\scr{E} \# H$ is an $\scr{H}$-cleft object with section $\sect_{\scr{E}} \ot \id$ and $[(\scr{E} \# H, \sect_{\scr{E}} \ot \id)] \in \cleft'(\scr{H})$.
	\end{lem}
	\begin{proof}
		For $e \in \scr{E}$ we denote $\coact_\scr{E} (e) = e_{(0)} \ot e_{(1)} \in \scr{E} \ot R$.
		$E:=\scr{E} \# H$ is a $\scr{H}$-comodule algebra, where the coaction $\rho_E : E \rightarrow E \ot \scr{H}$ for $e\in \scr{E}, h \in H$ is given by
		\begin{align*}
			\rho_E(e \ot h) = e_{(0)} \ot e_{(1)(-1)} h_{(1)} \ot e_{(1)(0)} \ot h_{(2)} \in \scr{E} \ot H \ot R \ot H \matdot
		\end{align*}
		Moreover $\sect_E := \sect_{\scr{E}} \ot \id : \scr{H} \rightarrow E$ defines a section, convolution invertible with inverse $(\biunit_{\scr{E}} \bicounit_R \ot \antip_H) \conv (\coninv{\sect_{\scr{E}}} \ot \biunit_H \bicounit_H)$. We also have for $r\in R$, $h,h',h'' \in H$
		\begin{align*}
			&\sect_E \left( (1 \ot h) ( r \ot h') (1 \ot h'') \right) 
			= \sect_E \left( h_{(1)} \cdot r \ot h_{(2)} h' h'' \right)
			\\=& \sect_\scr{E} (h_{(1)} \cdot r) \ot h_{(2)} h' h'' 
			= h_{(1)} \cdot \sect_\scr{E} ( r) \ot h_{(2)} h' h''
			\\=& (1 \ot h) (\sect_{\scr{E}}( r) \ot h') (1 \ot h'') 
			= \sect_E (1 \ot h) \sect_E( r \ot h') \sect_E(1 \ot h'') 
			\matcom
		\end{align*}
		since $\sect_{\scr{E}}$ is a $H$-module morphism. Finally if $e \ot h \in E^{\mathrm{co} \scr{H}}$, then
		\begin{align*}
			e_{(0)} \ot e_{(1)(-1)} h_{(1)} \ot e_{(1)(0)} \ot h_{(2)} = e \ot h \ot 1 \ot 1  \matdot
		\end{align*}
		Applying $ \id \ot \bicounit_H\ot \id \ot \bicounit_H$ to that relation yields $\bicounit_H(h) e \in \scr{E}^{\mathrm{co} R} = \fK$. Applying $\id \ot \bicounit_H\ot \bicounit_R \ot \id$ gives $e \ot h = \bicounit_H (h) e \ot 1 \in \fK$. Hence $E^{\mathrm{co} \scr{H}} = \fK$.
	\end{proof}
	
	\begin{prop} \label{prop_cleftcocyclebij}
		The maps
		\begin{align*}
		&\cocyclebij:  \cocyclesetR \rightarrow \cocyclesetHp, && \cocyclebij(\pi) \left( r \ot h , r' \ot h' \right) = \pi \left(r , h \cdot r' \right) \bicounit_H(h') \matcom
		\\ &\cleftbij:  \cleft(R) \rightarrow \cleft'(\scr{H}), && \cleftbij([(\scr{E},\sect)]) =  [(\scr{E} \# H, \sect \ot \id)] \matcom
		\end{align*}
		for all $\pi \in \cocyclesetR$, $r,r' \in R$, $h,h' \in H$, $[(\scr{E},\gamma)] \in \cleft(R)$, are well defined bijections. Moreover the following diagram commutes:
		\begin{align*}\begin{tikzcd}[ampersand replacement=\&]
		\cocyclesetR \arrow{d}{\alpha_R} \arrow{r}{\cocyclebij} \& \cocyclesetHp \arrow{d}{\alpha_{\scr{H}}} \\
		\cleft(R) \arrow{r}{\cleftbij} \& \cleft'(\scr{H}) 
		\end{tikzcd} \matdot \end{align*}
	\end{prop}
	\begin{proof}
		With Proposition \ref{prop_cleftbij_restrictsp} it is enough to show that the maps are well defined, that $\cocyclebij$ is bijective and that the diagram commutes.
		$\cleftbij$ is well defined by Lemma \ref{lem_cleftr_gives_clefth}. 
		If $\pi \in \cocyclesetR$, then $\cocyclebij(\pi)$ is convolution invertible with inverse
		\begin{align*}
		 \coninv{\cocyclebij}(\pi) \left( r \ot h , r' \ot h' \right) = \coninv{\pi} \left(r , h \cdot r' \right) \bicounit_H(h') \matcom
		\end{align*}
		for all $r,r' \in R$, $h,h' \in H$. $\fK \#_\pi R \in \scr{D}$ is by definition a $H$-module algebra. The multiplication of $(\fK \#_{\pi} R) \# H$ is precisely $\bimu_{\cocyclebij(\pi)}$, hence by Proposition \ref{prop_sigmul_assoc} $\cocyclebij(\pi) \in \cocyclesetH$. It is elementary to conclude $\cocyclebij(\pi) \in \cocyclesetHp$ and thus $\cocyclebij$ is well defined. We can also conclude that the diagram commutes. Finally we show that $\cocyclebij$ is bijective: It is clearly injective and if $\sigma \in \cocyclesetHp$, then Lemma \ref{lem_cocyclep_rel} implies that the map $\pi : R \ot R \rightarrow \fK$, $\pi(r,r') = \sigma(r\ot1 ,r' \ot 1)$ is convolution invertible (regarding $\brd$) with inverse $\coninv{\pi}(r,r') = \coninv{\sigma}(r\ot1 ,r' \ot 1)$. Since $\sigma \in \cocyclesetH$ we have for $x,y,z \in \scr{H}$:
		\begin{align*}
		\sigma(x_{(1)},y_{(1)}) \sigma(x_{(2)}y_{(2)}, z) &= \sigma(y_{(1)},z_{(1)}) \sigma(x,y_{(2)} z_{(2)}) \matdot
		\end{align*}
		Calculating this for $x=1 \ot h$, $y=r \ot 1$, $z=r' \ot 1$, $h \in H$, $r,r' \in R$ we obtain
		\begin{align*}
		\sigma(h_{(1)} \cdot r \ot 1, h_{(2)} \cdot r' \ot 1) = \bicounit_H(h) \sigma (r\ot1, r'\ot 1) \matcom
		\end{align*}		
		i.e. that $\pi$ is $H$-module morphism.
		Using $x = r \ot 1$, $y = r' \ot 1$, $z=r'' \ot 1$, $r,r',r'' \in R$ we obtain
		\begin{align*}
		&\sigma(r_{(1)} \ot 1, r_{(2)(-1)} \cdot {r'}_{(1)} \ot 1) \sigma(r_{(2)(0)} {r'}_{(2)} \ot 1, r'' \ot 1) 
		\\=& \sigma({r'}_{(1)} \ot 1, {r'}_{(2)(-1)} \cdot {r''}_{(1)} \ot 1) \sigma(r \ot 1, {r'}_{(1)(0)}  {r''}_{(2)} \ot 1) \matcom
		\end{align*}
		i.e. $\pi \in \cocyclesetR$. Now $\cocyclebij(\pi)=\sigma$, hence $\cocyclebij$ is surjective.
	\end{proof}
	
	\begin{cor} \label{cor_cleftcocyclebij}
		We have
		\begin{align*}
		\cleft' (\scr{H}) &= \lbrace [(\scr{E} \# H,\sect \ot \id)] \,|\, [(\scr{E},\sect)] \in \cleft (R) \rbrace 
		\\&= \lbrace [(k \#_{\cocyclebij(\pi)} \scr{H}, \id_\scr{H})] \,|\, \pi \in \cocyclesetR \rbrace \matdot
		\end{align*}
	\end{cor}
	
	\begin{defi}
		For $\sigma \in \cocyclesetH$, $\scr{H}_\sigma$ denotes the Hopf algebra that as a coalgebra is $\scr{H}$, with multiplication
		\begin{align*}
		\bimu_{\scr{H}_\sigma}(x,y) &= \sigma(x_{(1)} \ot y_{(1)}) \, x_{(2)} y_{(2)} \, \coninv{\sigma} (x_{(3)} \ot y_{(3)}) \matcom
		\end{align*}
		for all $x,y \in \scr{H}$.
	\end{defi}
	
	Assume $H$ to be finite-dimensional and cosemisimple and assume $R$ to be finite dimensional connected graded, with grading $R= \oplus_{k \in \ndN_0} R_{(k)}$. Observe that then $\scr{H}_{(k)} = R_{(k)} \ot H$, $k\ge 0$ defines a Hopf algebra grading on $\scr{H} = R \# H$.
	
	\begin{lem} \label{lem_gr_cocyclep}
		If $\sigma \in \cocyclesetHp$, then $\mathrm{gr} \scr{H}_\sigma \cong \scr{H}$.
	\end{lem}
	\begin{proof}
		Let $k,l \ge 1$. Since $R$ is connected we have for $r \in R_{(k)}$:
		\begin{align*}
			\bicomu_R (r) &= r\ot 1 + 1 \ot r + d \matcom & \text{where } d \in \oplus_{i=1}^{k-1} R_{(i)} \ot R_{(k-i)} \matdot
		\end{align*}
		Hence we get for $r \in R_{(k)}$, $h \in H$:
		\begin{align*}
			& (\bicomu_{\scr{H}} \ot \id) \bicomu_{\scr{H}} (r \ot h) 
			\\ =& ( r \ot h_{(1)} ) \ot (1 \ot h_{(2)}) \ot (1 \ot h_{(3)}) 
			\\&+ ( 1 \ot r_{(-1)} h_{(1)} ) \ot (r_{(0)} \ot h_{(2)}) \ot (1 \ot h_{(3)})
			 \\&+ ( 1 \ot (r_{(-1)} h_{(1)})_{(1)} ) \ot (1 \ot (r_{(-1)} h_{(1)})_{(2)}) \ot (r_{(0)} \ot h_{(2)}) + d \matcom
		\end{align*}
		where $d \in \oplus^{1\le i_1,i_2,i_3 \le k-1 \matcom}_{i_1+i_2+i_3=k } \scr{H}_{(i_1)} \ot \scr{H}_{(i_2)} \ot \scr{H}_{(i_3)}$. Hence for $r \in R_{(k)}$, $r' \in R_{(l)}$, $h,h' \in H$:
		\begin{align*}
			& \bimu_{\scr{H}_\sigma}(r \ot h, r' \ot h') 
			\\ =& \sigma(1 \ot r_{(-1)} h_{(1)}, 1 \ot {r'}_{(-1)} {h'}_{(1)}) \, r_{(0)} (h_{(2)} \cdot {r'}_{(0)}) \ot h_{(3)} {h'}_{(2)} \\& \coninv{\sigma}(1 \ot h_{(4)}, 1 \ot {h'}_{(3)}) + d \matcom
		\end{align*}
		where $d \in \oplus_{i=0}^{k+l-1} \scr{H}_{(i)}$. Now since $\sigma \in \cocyclesetHp$ we obtain
		\begin{align*}
			& \sigma(1 \ot r_{(-1)}  h_{(1)}, 1 \ot {r'}_{(-1)} {h'}_{(1)}) \, r_{(0)}  (h_{(2)} \cdot {r'}_{(0)}) \ot h_{(3)} {h'}_{(2)} \\& \coninv{\sigma}(1 \ot h_{(4)}, 1 \ot {h'}_{(3)})
			\\=& \bicounit_{H} (r_{(-1)}  h_{(1)}) \bicounit_{H} ({r'}_{(-1)}  {h'}_{(1)}) ( r_{(0)} (h_{(2)} \cdot {r'}_{(0)}) \ot h_{(3)} {h'}_{(2)} ) \\& \bicounit_{H} (h_{(4)}) \bicounit_{H}({h'}_{(3)})
			\\=&   r (h_{(1)} \cdot r') \ot h_{(2)} h' \matdot
		\end{align*}
		Thus the multiplication of $\mathrm{gr} \scr{H}_\sigma$ coincides with the one in $\scr{H}$, giving $\mathrm{gr} \scr{H}_\sigma \cong \scr{H}$.
	\end{proof}
	
		When one is interested in finding all liftings of a coradically graded Hopf algebra, i.e. in our terms all filtered Hopf algebras $L$, such that $\mathrm{gr} L \cong \scr{H}$, the problem reduces to finding all two-cocycles $\sigma : \scr{H} \ot \scr{H} \rightarrow \fK$, such that $\mathrm{gr} \scr{H}_\sigma \cong \scr{H}$  (\cite{guidedtour}, section 3.1). This is given by finding all $\scr{H}$-cleft objects $E$, such that $\mathrm{gr} \scr{H}_{\sigma_{\sect}} \cong \scr{H}$ for some section $\sect:\scr{H} \rightarrow E$. The set of ($\ol{H}$-comodule algebra) isomorphism classes $[E]$ of $\ol{H}$-cleft objects $E$ with this property is often denoted $\cleftcl'(\scr{H})$.
	
	We can describe $\cleftcl'(\scr{H})$ by $\cleft'(\scr{H})$:
	\begin{prop} \label{cor_cleftcocyclebij2}
		The map
		\begin{align*}
		\cleft'(\scr{H}) \rightarrow \cleftcl'(\scr{H}), \, [(E,\sect)] \mapsto [E]
		\end{align*}
		is surjective.
	\end{prop}
	\begin{proof}
		The map is well defined: If $[(E,\sect)] \in \cleft'(\scr{H})$, then $\sigma_{\sect} \in \cocyclesetHp$ by Lemma \ref{prop_cleftbij_restrictsp} and thus by Lemma \ref{lem_gr_cocyclep} $\mathrm{gr} \scr{H}_{\sigma_{\sect}} \cong \scr{H}$ and $[E] \in \cleftcl'(\scr{H})$.
		Now if $A \in \cleftcl'(\scr{H})$, then there exists a representative $E \in A$ and a section $\sect : \scr{H} \rightarrow E$, such that
		\begin{align*}
			\sect ( (1 \ot h) (r \ot h') (1 \ot h'') ) &= \sect (1 \ot h) \sect (r \ot h') \sect (1 \ot h') \matcom
		\end{align*}
		for all $r \in R, h,h',h'' \in H$ (\cite{MR3133699}, Proposition 5.8), i.e. $[(E,\sect)] \in \cleft'(\scr{H})$ is a preimage of $[E]=A$. 
	\end{proof}
	
	Proposition \ref{cor_cleftcocyclebij2} and Corollary \ref{cor_cleftcocyclebij} imply:
	\begin{cor} \label{cor_cleftcocyclebij3}
		We have
		\begin{align*}
		\cleftcl' (\scr{H}) &= \lbrace [\scr{E} \# H] \,|\, [(\scr{E},\sect)] \in \cleft (R) \rbrace 
		= \lbrace [k \#_{\cocyclebij(\pi)} \scr{H}] \,|\, \pi \in \cocyclesetR \rbrace \matdot
		\end{align*}
	\end{cor}
	
\newcommand{\etalchar}[1]{$^{#1}$}
\providecommand{\bysame}{\leavevmode\hbox to3em{\hrulefill}\thinspace}
\providecommand{\MR}{\relax\ifhmode\unskip\space\fi MR }
\providecommand{\MRhref}[2]{%
	\href{http://www.ams.org/mathscinet-getitem?mr=#1}{#2}
}
\providecommand{\href}[2]{#2}

\end{document}